\documentclass{amsart}

\usepackage[dvipsnames]{xcolor}
\usepackage{amsmath}
\usepackage{enumerate}
\usepackage{enumitem}
\usepackage{amssymb}
\usepackage{extarrows}
\usepackage{amsthm}
\usepackage{thmtools}
\usepackage{mathrsfs}
\usepackage[colorlinks, linkcolor = SeaGreen, citecolor = SeaGreen]{hyperref}
\usepackage{bm}
\hypersetup{colorlinks = true}

\newtheorem{thm}{Theorem}[section]

\newtheorem{lem}[thm]{Lemma}

\newtheorem{ques}[thm]{Question}
\theoremstyle{definition}

\theoremstyle{remark}

\newcommand{\CC}{\mathbb{C}}
\newcommand{\DD}{\mathbb{D}}

\newcommand{\wind}{\mathrm{wind}}

\title[The Douglas question]{The Douglas question for functions of the form $\bm{\overline{z}+h}$ with $\bm{h}$ in the disk algebra}

\author[Wang]{Jiawei Wang}
\address{School of Mathematical Sciences, Inner Mongolia University, Hohhot, 010021, P. R. China}
\email{jiaweiwang1516@163.com}

\author[Zhao]{Xianfeng Zhao}
\address{College of Mathematics and Statistics, Chongqing University, Chongqing, 401331, P. R. China}
\email{xianfengzhao@cqu.edu.cn}



\begin{document}
	\begin{abstract}
		We prove that the Toeplitz operator $T_{\overline{z}+h}$  is invertible on the Bergman space provided that
 $|\overline{z}+h|$ is bounded below by some positive constant on the open unit disk,
 where $h$ is in the disk algebra. This provides a  general class of  harmonic functions for which the answer to the Douglas question on the Bergman space is affirmative.
	\end{abstract}
	
	\maketitle
	
	\section{Introduction} \label{Introduction}
	
Let $dA$ denote the Lebesgue area measure on the open unit disk $\mathbb D$ in the complex plane $\mathbb C$, normalized so that the measure of the disk
$\mathbb D$ is $1$. In the sequel, we denote by $L^2(\mathbb D)=L^2(\mathbb D, dA)$ the Hilbert space of square integrable functions with respect to $dA$.
The Bergman space $L_a^2$ is the closed subspace of $L^2(\mathbb D)$ consisting of analytic functions on $\mathbb D$, which
is also a Hilbert space with the inner product
$$\langle f, g \rangle=\int_{\mathbb D} f(z)\overline{g(z)}dA(z), \ \ \ \  f, g\in L_a^2.$$
Let $P$ be the orthogonal projection  from $L^2(\mathbb D)$ onto the Bergman space $L_a^2$. Let $L^\infty(\mathbb D)= L^\infty(\mathbb D, dA)$ denote the Banach algebra of essentially bounded functions on $\mathbb D$ with respect to the measure $dA$. For $\varphi \in L^\infty(\mathbb D)$, the  Toeplitz operator $T_{\varphi}$ with symbol $\varphi$ on the Bergman space  is  defined by $$T_\varphi f=P(\varphi f)$$
for $f\in L_a^2$.

Let $\varphi\in L^{\infty}({\mathbb D})$.  The Berezin transform of the Toeplitz operator
$T_\varphi$ is defined by
\begin{align*}
\widetilde{T_{\varphi}}(z)
=\langle T_{\varphi}k_z,k_z \rangle=\int_{\mathbb D}\varphi(w)|k_z(w)|^2dA(w),
\end{align*}
where $k_z$ is the normalized Bergman reproducing kernel for $L_a^2$ given by $$k_z(w)=\frac{1-|z|^2}{(1-\overline{z}w)^2}.$$
Usually, the  Berezin transform of
$T_\varphi$ is denoted by $\widetilde{\varphi}$ for simplicity, which we also call the Berezin transform of $\varphi$. For more information about Toeplitz operators on the Bergman space and the Berezin transform,
we refer to \cite{Ax} and \cite{Zhu}.

The significance of spectral theory in operator theory derives from its role as a unified framework for investigating the invertibility and dynamical behavior of operators in terms of their spectra as numerical invariants. Among the various topics in Toeplitz operator theory, the characterization of invertibility via properties of symbols stands as a particularly significant and challenging problem. See \cite{Dou}, \cite{Lue} and \cite{Zhu} for detailed discussions. The original Douglas question is an influential open question concerning the invertibility problem of Toeplitz operators on the Hardy space (see \cite{Dou1}):
\begin{ques}\label{Douglas-H}
Suppose that $\varphi$ is an essentially bounded function on the unit circle $\partial \mathbb D$, and  the harmonic extension \emph{(}poisson integral\emph{)} of $\varphi$ is bounded below by some positive constant on the unit disk $\mathbb D$, i.e., there exists a constant $\delta>0$ such that $|\widehat{\varphi}(z)|\geqslant \delta$ for all $z\in \mathbb D$, where
$$\widehat{\varphi}(z):=\frac{1}{2\pi}\int_0^ {2\pi}\varphi (e^{\mathrm{i}\theta}) \frac{1-|z|^2}{|1-ze^{-\mathrm{i}\theta}|^2}d\theta, \ \ \ \ z\in \mathbb D.$$
Then is the Toeplitz operator with symbol $\varphi$  invertible on the Hardy  space?
\end{ques}

The original Douglas question was studied by many authors. For instance, Tolokonnikov \cite{Tol} and Nikolskii \cite{Nik} gave affirmative answers to Question  \ref{Douglas-H} under certain conditions. In 2002,  by employing the martingale limit of a family of step functions, Wolff \cite{Wol} constructed a bounded function whose harmonic extension is bounded below on the unit disk, but for which the corresponding Toeplitz operator is not invertible on the Hardy space, thereby providing a negative answer to the original Douglas question.

  Note that the harmonic extension of an essentially bounded function $\varphi$ defined on $\partial \mathbb D$ is equal to the Berezin transform of the Toeplitz operator with symbol $\varphi$ on the Hardy space.
This observation leads to a natural analogue of the original Douglas question in the setting of the Bergman space:
\begin{ques} [\emph{The Douglas question on the Bergman space}]\label{Douglas-B}
Let $\varphi$ be in $L^\infty (\mathbb D)$. Then is $T_{\varphi}$ invertible  on the Bergman space $L_a^2$ if
  $|\widetilde{\varphi}(z)| \geqslant \delta$ for some positive constant
 $\delta$
 and for all $z\in \mathbb D$?
 \end{ques}

In view of the fact that many properties of Toeplitz operators with harmonic symbols on the Bergman space are analogous to those of Toeplitz operators on the Hardy space, scholars have paid particular attention to the case of harmonic symbols in Question \ref{Douglas-B}. This is the most natural, important, and difficult case of the Douglas question on the Bergman space.

 We first recall the existing results concerning the Douglas question on the Bergman space in the case of harmonic symbols. In 1979, McDonald and Sundberg \cite{MS} showed that for a bounded real-valued harmonic function $\varphi$, the spectrum of the Toeplitz operator $T_\varphi$ is equal to the  interval $[\mathrm{inf}(\varphi), \mathrm{sup}(\varphi)]$.  Since the Berezin transform of a bounded harmonic function is identical to the function itself, the result of \cite{MS} provides  a positive answer to the Douglas question for real-valued harmonic symbols.
More than thirty years later,  Zhao and Zheng \cite{ZZ} established that the Toeplitz operator $T_{\overline{z}+(az+b)}$ is invertible on the Bergman space $L_a^2$ if and only if its symbol $\overline{z}+(az+b)$ is invertible in $L^\infty(\mathbb D)$. Later, Guo, Zhao and Zheng \cite{GZZ} proved that $T_{\overline{z}+p}$ is invertible on $L_a^2$ provided that $|\overline{z}+p|$ is bounded below by some positive constant on the unit disk $\mathbb D$, where $p$ is an analytic polynomial of degree at most $2$. However,  when $p$ is an analytic polynomial with real coefficients,  Guan and Zhao \cite{GZ} proved that $T_{\overline{z}+p}$ is invertible on the Bergman space if  $\overline{z}+p$ is invertible in $L^\infty(\mathbb D)$. In 2022, Yoneda \cite{Yo} proved that
$T_{\alpha g+\beta\overline{g}}$ is invertible on $L_a^2$ if
$\inf\limits_{z\in \mathbb D}|\alpha g(z)+\beta\overline{g}(z)|>0$, where $\alpha, \beta\in \mathbb C$ and $g\in H^\infty$, the space of bounded analytic functions on $\mathbb D$. Recently, Cui, Lu, Yang and Zu \cite{CLYZ} established that, for a positive integer $m$  and $\alpha,\beta\in \mathbb C$, the spectrum of the Toeplitz operator $T_{\overline{z}^m+(\alpha z^m+\beta)}$ is precisely the closure of the range of $\overline{z}^m+(\alpha z^m+\beta)$.  Therefore, the results of \cite{ZZ, GZZ, GZ}, \cite{Yo} and \cite{CLYZ}  all provide affirmative answers to the Douglas question on the Bergman space for some special harmonic symbols.
For general harmonic symbols, Zhao and Zheng \cite{ZZ1, ZZ2} showed that $T_{\varphi}$ is invertible on the Bergman space if
$\varphi$ is a bounded harmonic function on $\mathbb D$ satisfying
$$|\varphi(z)|\geqslant \delta \|\varphi\|_\infty\ \ \  \ (z\in \mathbb D)$$
for some $\delta\in (\frac{2\sqrt{2}}{3}, 1)$. Using the same idea as in  \cite{ZZ1},
 Yoneda \cite{Yo}
obtained an analogous result, namely that if $g$ and $h$ belong to the Bloch space and
$$|h(z)+\overline{g}(z)|>4.27\max\{\|h\|_{\mathcal B}, \|g\|_{\mathcal B}\},$$
then
$T_{h+\overline{g}}$ is invertible on $L_a^2$, where  $\|h\|_{\mathcal B}$ and $\|g\|_{\mathcal B}$ denote the norms of $h$ and $g$ in the Bloch space, respectively.

In 2016, Zhao and Zheng \cite{ZZ1} established that the answer to Question \ref{Douglas-B} is also affirmative for bounded nonnegative symbols. However, for symbols that are not nonnegative or not harmonic, the answer to the Douglas question on the Bergman space $L_a^2$ may be negative. In fact, it was shown in \cite{ZZ1} that there exists a continuous radial function $\varphi$ on the closed disk $\overline{\mathbb D}$ such that
$\varphi$ and $\widetilde{\varphi}$ are both invertible in $L^\infty (\mathbb D)$, while the Toeplitz operator $T_\varphi$ is not invertible on $L_a^2$.
Additionally, for the investigation of the Douglas question for Toeplitz operators with measure symbols, one can consult \cite{Cu} and \cite{CLZ} for more details.

Although the Douglas question on the Bergman space $L_a^2$ remains open for harmonic symbols, significant progress has been achieved in the study of the invertibility of Toeplitz operators with various symbols on $L_a^2$. In 1981, Luecking \cite{Lue} obtained a necessary and sufficient condition for $T_{\varphi}$ to be  invertible on $L_a^2(\mathbb D)$ in the case that
$\varphi$ is nonnegative on $\mathbb D$.  Based on Luecking's results,  Faour \cite{Fao}  gave
a necessary condition for $T_{\varphi}$ to be  invertible if  $\varphi$ is  continuous on the closed unit disk and satisfies that
$|\varphi(z_1)|\geqslant |\varphi(z_2)|$ whenever  $|z_1|\leqslant |z_2|$. In a more general setting, using the Berezin transform and atomic decomposition, Karaev \cite{Kar} obtained sufficient conditions for bounded linear operators to be invertible on the Bergman space. By applying Karaev's results, G\"{u}rdal and S\"{o}hret
\cite{Gur} established  a sufficient condition on the invertibility of Toeplitz operators with bounded symbols. By employing the theory of differential equations with holomorphic coefficients, Tikaradze  \cite{Tik} characterized the invertibility of the Toeplitz operator $T_\varphi$ whose symbol is given by $\varphi(z)=az^n(z-w)^m+(\overline{z}-1/w)^m$, where $a\in\mathbb C$, $w\in\mathbb D\backslash\{0\}$, and  $m,n\in \mathbb N$. Under a certain geometric assumption, \v{C}u\v{c}kovi\'{c} and Taskinen \cite{CT} recently showed that  the Toeplitz operator $T_\varphi$ is invertible on the Bergman space if and only if the Berezin transform of $|\varphi|$ is invertible in $L^\infty(\mathbb D)$. Furthermore, they studied a class of general harmonic polynomials and characterized the invertibility of the corresponding Toeplitz operators.

Let $C(\overline{\mathbb D})$ denote the space of continuous functions on the closed unit disk $\overline{\mathbb D}$. Let $A(\mathbb D)$ denote the disk algebra, i.e., $A(\mathbb D)=H^\infty(\mathbb D)\cap C(\overline{\mathbb D})$. In this paper, we show in Theorem \ref{Main} that the Toeplitz
operator $T_{\overline{z}+h}$ is invertible on the Bergman space $L_a^2$ whenever
the symbol $\overline{z}+h$ does not vanish on the closed disk $\overline{\mathbb D}$, where $h\in A(\mathbb D)$. This yields an affirmative answer to the Douglas question on the Bergman space for such a class of harmonic functions.

\section{Preliminaries} \label{Preliminaries}

In this section, we recall some basic facts about Toeplitz operators on the Bergman space $L_a^2$ that will be needed in the next section. Let us begin with the following lemma about the  essential spectra and Fredholm index of Toeplitz operators with continuous symbols on $L_a^2$.
	
	\begin{lem}[{\cite[Theorem 24]{StrZ}}]\label{FI}
		Suppose that $\varphi$ is continuous on the closed disk  $\overline{\DD}$. Then the essential spectrum of $T_\varphi$ \textup{(}denoted by $\sigma_{\mathrm{e}} (T_{\varphi}))$ is  equal to $\varphi (\partial \DD)$, and the Fredholm index is given by
		\begin{align*}
			\mathrm{index}{(T_{\varphi} - \lambda I)} = -\wind (\varphi (\partial \DD), \lambda)
		\end{align*}
for each $\lambda \in \CC\backslash\varphi (\partial \DD)$,
		where  $I$ stands for the identity operator on the Bergman space and  $\wind (\varphi (\partial \DD), \lambda)$ is the winding number of the curve $\varphi (\partial \DD)$  with respect to the point $\lambda$.
	\end{lem}

The next lemma is quite useful for us to study the spectral properties of Toeplitz operators with harmonic  symbols on the Bergman space.

\begin{lem}[{\cite[Lemma 2.1]{SZ}}]\label{T}
For $f$ in the Bergman space $L_a^2$, we have
\begin{align*}
T_{\overline{z}}f(z)=\frac{1}{z^2} \int_0^{z} w f'(w)dw.
\end{align*}
\end{lem}

We end this section with the following lemma, which plays a key role in our study of the point spectra of Toeplitz operators with harmonic symbols on the Bergman space.

\begin{lem}[{\cite[Theorem 2.4]{GZZ}}]\label{key lemma}
Let $h$ be a function in the disk algebra $A(\mathbb D)$. Suppose that $\lambda$ is a complex number not in the essential spectrum of the Toeplitz operator $T_{\overline{z}+h}$. Then $\lambda$ is an eigenvalue of  $T_{\overline{z}+h}$ if and only if  either
$1+z[h(z)-\lambda]$ does not vanish on the unit disk or $1+z[h(z)-\lambda]$ has finitely many simple zeros $\big\{z_1, \cdots, z_k\big\}$ in the open unit disk which satisfy
\begin{align*}
z_j^2h'(z_j)=\frac{n_j+2}{n_j+1}
\end{align*}
 for some integer $n_j\in \big\{0, 1, 2, \cdots\big\}$ with $j=1, 2, \cdots, k$.
\end{lem}

\section{Main Results}

The main result of this paper is:

\begin{thm}	\label{Main}
Let $h$ be in the disk algebra $A(\mathbb D)$. If there exists a constant $\delta>0$ such that $$|\overline{z}+h(z)|\geqslant \delta$$
for all  $z\in \mathbb D$,
then the Toeplitz operator $T_{\overline{z}+h}$ is invertible on the Bergman space $L_a^2$.
\end{thm}

Before presenting the proof of the above theorem, we require one more lemma which is helpful for us to characterize certain analytic properties of the eigenvectors of the Toeplitz operator $T_{\overline{z}+h}$.

\begin{lem}\label{C}
Let $h$ be in the disk algebra $A(\mathbb D)$. Suppose that $\overline{z}+h$ is invertible in $L^\infty(\mathbb D)$.
If $0$ is an eigenvalue of the Toeplitz operator $T_{\overline{z}+h}$, then the corresponding eigenvector $f$ belongs to
$A(\mathbb D)$.
\end{lem}

\begin{proof}
Since  the origin  is not contained in the closure of
the range of $\overline{z}+h$, we deduce  by means of the homotopy theory that $$\mathrm{wind}\big([\overline{z}+h(z)]|_{\partial \mathbb D}, 0\big)=0.$$
Moreover, since $\overline{z}+h$ does not vanish on the unit circle $\partial \mathbb D$, we conclude by Lemma \ref{FI} that $0$ is not in the essential spectrum of $T_{\overline{z}+h}$ and  the Fredholm index  of $T_{\overline{z}+h}$ is given by
$$0=\mathrm{index}(T_{\overline{z}+h})=-\mathrm{wind}\big([\overline{z}+h(z)]|_{\partial \mathbb D}, 0\big)=-\mathrm{wind}\Big(\frac{1+zh(z)}{z}\Big|_{\partial \mathbb D}, 0\Big).$$
It follows from the argument principle that
$1+zh(z)$ has exactly one simply zero $z_0$ in $\mathbb D$. Then we have by Lemma \ref{key lemma} that
\begin{align}\label{n0}
 z_0^2 h'(z_0)=\frac{n_0+2}{n_0+1}
 \end{align}
for some integer $n_0\geqslant 0$.

Clearly, $f$ is an eigenvector of $T_{\overline{z}+h}$  corresponding to $0$ if and only if
$T_{\overline{z}}f=-hf.$
By Lemma \ref{T}, we obtain the following integral equation:
$$\frac{1}{z^2}\int_0^z wf'(w)dw=-h(z)f(z).$$
Multiplying both sides of the above equation by $z^2$ and then taking derivatives yields the following first order differential equation:
\begin{align}\label{diff}
\big[1+zh(z)\big]f'(z)=-\big[2h(z)+zh'(z)\big]f(z).
\end{align}
Therefore, $f$ is an eigenvector of the Toeplitz operator $T_{\overline{z}+h}$  corresponding to $0$ if and only if $f\in L^2_a\backslash\{0\}$ satisfies Equation (\ref{diff}).

Define
$$g(z)=\mathrm{exp}\left\{-\int_{0}^z \bigg(\frac{2h(w)+wh'(w)}{1+wh(w)}+\frac{n_0}{w-z_0}\bigg)dw \right\}$$
and
$$F(z)=g(z)(z-z_0)^{n_0},$$
where $n_0$ comes from (\ref{n0}).  Next, we are going to show that $F$ is an eigenvector of $T_{\overline{z}+h}$ corresponding to $0$. First, we show that $F\in A(\mathbb D)\subset L_a^2$. Indeed, since
$z_0$ is the unique simple zero of $1+wh(w)$ in $\mathbb D$ and $1+wh(w)$ does not vanish on $\partial \mathbb D$,  we can write
$$1+wh(w)=(w-z_0)\psi(w),$$
where $\psi\in A(\mathbb D)$ and $\psi(z)\neq 0$ for $|z|\leqslant 1$.
Let
$$Q(w)=\frac{h(w)}{1+wh(w)}+\frac{n_0+1}{w-z_0}.$$
Since $h(z_0)\neq 0$, we have  by the Cauchy residue theorem that
\begin{align*}
\mathrm{Res}(Q; z_0)&=\frac{h(z_0)}{\psi(z_0)}+n_0+1=\frac{h(z_0)}{h(z_0)+z_0h'(z_0)}+n_0+1\\
&=\frac{z_0 h(z_0)}{z_0 h(z_0)+z_0^2 h'(z_0)}+n_0+1\\
&=\frac{-1}{-1+\frac{n_0+2}{n_0+1}}+n_0+1=0.
\end{align*}
From the preceding discussion, we conclude  that $Q\in A(\mathbb D)$. Moreover, since
\begin{align*}
\frac{2h(w)+wh'(w)}{1+wh(w)}+\frac{n_0}{w-z_0}&=\frac{[1+wh(w)]'}{1+wh(w)}+\frac{h(w)}{1+wh(w)}+\frac{n_0}{w-z_0}\\
&=\frac{\psi'(w)}{\psi(w)}+\frac{1}{w-z_0}+Q(w)-\frac{n_0+1}{w-z_0}+\frac{n_0}{w-z_0}\\
&=\frac{\psi'(w)}{\psi(w)}+Q(w),
\end{align*}
we get that
\begin{align*}
g(z)=\mathrm{exp}\left\{-\int_{0}^z \frac{\psi'(w)}{\psi(w)}dw \right\}\mathrm{exp}\left\{-\int_{0}^z Q(w)dw \right\}.
\end{align*}
Denote
$$H(z)=\mathrm{exp}\left\{-\int_{0}^z \frac{\psi'(w)}{\psi(w)}dw \right\}.$$
Then it is easy to show that
$$(\psi H)'(z)=0, \ \ \ \ z\in \mathbb D,$$
to obtain
$$\psi(z)H(z)=\psi(0)H(0)=\psi(0), \ \ \ \ z\in \mathbb D.$$
Therefore, $H(z)=\frac{\psi(0)}{\psi(z)}$ and
\begin{align*}
g(z)
=\frac{\psi(0)}{\psi(z)}\mathrm{exp}\left\{-\int_{0}^z Q(w)dw \right\}.
\end{align*}
It follows that $g$ belongs to $A(\mathbb D)$, and so does $F$.

To show that $F$ is an
eigenvector of $T_{\overline{z}+h}$ corresponding to $0$, it remains to verify that $F$ satisfies (\ref{diff}). From the expressions of $g$ and $F$ we have that
\begin{align*}
\frac{F'(z)}{F(z)}&=\frac{n_0}{z-z_0}+\frac{g'(z)}{g(z)}\\
&=\frac{n_0}{z-z_0}-\bigg(\frac{2h(z)+z h'(z)}{1+zh(z)}+\frac{n_0}{z-z_0}\bigg)\\
&=-\frac{2h(z)+z h'(z)}{1+zh(z)},
\end{align*}
which gives
$$[1+zh(z)]F'(z)=-[2h(z)+zh'(z)]F(z).$$
This shows that $F$ is the
eigenvector of $T_{\overline{z}+h}$ corresponding to $0$.

Finally, since the solution space of a first-order linear differential equation is one-dimensional, the function $f$ in the hypothesis must be a constant multiple of $F$. Consequently, $f$ is also in $A(\mathbb D)$. This finishes the proof of the lemma.
\end{proof}

Now we are ready to prove Theorem \ref{Main}.

\begin{proof}[Proof of Theorem \ref{Main}]
Suppose that $\overline{z}+h$ is invertible in $L^\infty(\mathbb D)$  but the Toeplitz operator $T_{\overline{z}+h}$ is not invertible on the Bergman space. Using the same argument as the one used in the proof of Lemma \ref{C}, we conclude  that $T_{\overline{z}+h}$ is a Fredholm operator on $L_a^2$ and its Fredholm index is $0$. Hence, we have $\mathrm{ker}(T_{\overline{z}+h})\neq \{0\}$.

Let $f$ be a nonzero function in $L_a^2$ such that $T_{\overline{z}+h}f=0$. As shown in the proof of Lemma \ref{C}, $f$ satisfies that
$$\int_0^z wf'(w)dw+z^2f(z)h(z)=0.$$
By applying integration by parts, we get that
$$\int_0^z f(w)dw=z[1+zh(z)]f(z).$$
Define $$G(z)=\begin{cases}
\frac{\int_0^z f(w)dw}{z}, & z\neq 0,\vspace{2mm}\\
f(0), & z=0.\end{cases}$$
Then $G$ is analytic on $\mathbb D$ and
$$f(z)=G(z)+zG'(z), \ \ \ \ z\in \mathbb D.$$
and
$$G(z)=[1+zh(z)]f(z)=[1+zh(z)][G(z)+zG'(z)].$$

Next, we are going to show that the function $G$ satisfies that
$G(0)\neq 0$ and $G'(0)\neq 0$.
Indeed, since $f(0)=G(0)$, we need only to show that $f(0)\neq 0$. Otherwise, there exists
an integer $k\geqslant 1$ such that
$$f(z)=c_k z^k+c_{k+1}z^{k+1}+\cdots$$
with $c_k\neq 0$. It follows that
$$\int_0^z f(w)dw=\frac{c_k}{k+1}z^{k+1}+\frac{c_{k+1}}{k+2}z^{k+2}+\cdots,  \ \ \ \ |z|<1.$$
On the other hand,
$$\int_0^z f(w)dw=zG(z)=z[1+zh(z)]f(z)=zf(z)+z^2h(z)f(z)=c_{k}z^{k+1}+\cdots.$$
By comparing coefficients, we obtain
$$\frac{c_k}{k+1}=c_{k},$$
to get that $k=0$ since $c_k\neq 0$. This contradicts the fact that $k\geqslant1$, which
gives  $G(0)\neq 0$.

In order to show $G'(0)\neq 0$, we use
\begin{align} \label{G}
G(z)=[1+zh(z)][G(z)+zG'(z)]
\end{align}
to obtain that
$$\frac{G'(z)}{G(z)}=\frac{1}{z}\Big[\frac{1}{1+zh(z)}-1\Big]=
-\frac{h(z)}{1+zh(z)}.$$
It follows that
$$\frac{G'(0)}{G(0)}=-h(0)\neq 0,$$
since $\overline{z}+h(z)\neq 0$ for all $|z|\leqslant 1$. Therefore, $G'(0)\neq 0$.

Now we consider the following real-valued function:
$$\Psi(z)=(1-|z|^2)|G(z)|^2,  \ \ \ \  z\in  \overline{\mathbb D}.$$
According to Lemma \ref{C}, we obtain  that $f\in A(\mathbb D)$, and so does $G$. This implies that $\Psi\in C (\overline{\mathbb D})$. Moreover, $\Psi|_{\partial \mathbb D}=0$ and
$\Psi(0)=|G(0)|^2>0$.
Therefore, there exists a  point $\alpha\in \mathbb D$ such that $\Psi$ attains its maximum on $\overline{\mathbb D}$ and
$$\frac{\partial \Psi}{\partial \overline{z}}(\alpha)=0, \ \ \ \  \Psi(\alpha)\geqslant\Psi(0)>0.$$
This implies that $G(\alpha)\neq 0.$

Taking the partial derivative of $\Psi$
 with respect to $\overline{z}$ gives
 $$\frac{\partial \Psi}{\partial \overline{z}}=-z|G(z)|^2+(1-|z|^2)G(z)\overline{G'(z)}.$$
Evaluating at $z=0$, we have
$$\frac{\partial \Psi}{\partial \overline{z}}(0)=G(0)\overline{G'(0)}\neq 0,$$
which means that $0$ is not a maximum point of $\Psi$. Thus, we derive that  $0<|\alpha|<1$.
However, we obtain by $\frac{\partial \Psi}{\partial\overline{ z}}(\alpha)=0$ that
$$-\alpha |G(\alpha)|^2+(1-|\alpha|^2)G(\alpha) \overline{G'(\alpha)}=0.$$
This gives
$$\overline{\alpha}=(1-|\alpha|^2) \frac{G'(\alpha)}{G(\alpha)},$$
since $G(\alpha)\neq 0$. Furthermore,
$$|\alpha|^2=(1-|\alpha|^2) \frac{\alpha G'(\alpha)}{G(\alpha)},$$
which is equivalent to
$$\frac{|\alpha|^2}{1-|\alpha|^2}=\frac{\alpha G'(\alpha)}{G(\alpha)}.$$
Using (\ref{G}) again, we have
$$1+\alpha h(\alpha)=\frac{G(\alpha)}{G(\alpha)+\alpha G'(\alpha)}=\frac{1}{1+\frac{\alpha G'(\alpha)}{G(\alpha)}}=\frac{1}{1+\frac{|\alpha|^2}{1-|\alpha|^2}}=1-|\alpha|^2,$$
which yields that
$$\overline{\alpha}+h(\alpha)=0.$$
But this is impossible, since $|\overline{z}+h|$ is bounded below  by $\delta$ on $\mathbb D$.
The contradiction yields that the Toeplitz operator $T_{\overline{z}+h}$ is invertible on
the Bergman space provided that $\overline{z}+h$ is invertible in $L^\infty(\mathbb D)$, completing the proof.
\end{proof}

 	\vspace{3mm}

\subsection*{Acknowledgments}
We thank Professor Dechao Zheng  for many useful discussions and suggestions. This work was supported by  National Natural Science Foundation of China (12371125) and  Chongqing Natural Science Foundation
 (CSTB2024NSCQ-MSX0177).

\end{document}